\documentclass{amsart}

\usepackage{amsmath,amsfonts,amssymb,bbm,MnSymbol,xcolor,array}

\newtheorem{theorem}{Theorem}
\newtheorem{lemma}[theorem]{Lemma}
\newtheorem{proposition}[theorem]{Proposition}
\newtheorem{corollary}[theorem]{Corollary}

\theoremstyle{definition}

\theoremstyle{remark}

\newcommand{\newword}[1]{\textbf{\emph{#1}}}

\newcommand{\Sym}{\mathfrak{S}}
\newcommand{\GL}{\mathrm{GL}}

\newcommand{\CC}{\mathbb{C}}
\renewcommand{\SS}{\mathbb{S}}

\newcommand{\Wh}{\mathrm{Lie}}

\newcommand{\Sp}{\mathrm{Sp}}

\newcommand{\ohat}{\hat{0}}

\newcommand{\Wr}{\mathrm{Wr}}

\newcommand{\gbar}{\bar{g}}

\newcommand{\sdag}{s^{\dagger}}

\newcommand{\Rep}{\mathrm{Rep}}
\newcommand{\Ind}{\mathrm{Ind}}
\newcommand{\Res}{\mathrm{Res}}

\newcommand{\ch}{\mathrm{ch}}

\newcommand{\Fine}{\mathrm{Fine}}
\newcommand{\Coarse}{\mathrm{Coarse}}
\newcommand{\Shape}{\mathrm{Shape}}
\newcommand{\DSh}{D_{\mathrm{Sh}}} 
\newcommand{\One}{\mathbbm{1}}
\newcommand{\Tr}{\mathrm{Tr}}
\newcommand{\mmu}{\mathfrak{m}} 
\newcommand{\meq}{\mmu_{\mathrm{eq}}}

\usepackage{xcolor}
\usepackage[colorinlistoftodos]{todonotes}

\begin{document}


\title[Spechts are alternating in Schurs]{Specht modules decompose as alternating sums of restrictions of Schur modules}  

\author[S. H. Assaf]{Sami H. Assaf}
\address{Department of Mathematics, University of Southern California, 3620 S. Vermont Ave., Los Angeles, CA 90089-2532, U.S.A.}
\email{shassaf@usc.edu}
\thanks{SHA supported by NSF DMS-1763336; DES supported by NSF DMS-1600223.}

\author[D. E. Speyer]{David E. Speyer}
\address{Department of Mathematics, University of Michigan, 530 Church St., Ann Arbor, MI 28109-1043, U.S.A.}
\email{speyer@umich.edu}

\subjclass[2010]{Primary 20C15; Secondary 20C30, 05E05, 05E10}



\keywords{Schur modules, Specht modules, plethysm, Kronecker coefficients}

\begin{abstract}
  Schur modules give the irreducible polynomial representations of the general linear group $\mathrm{GL}_t$. Viewing the symmetric group $\mathfrak{S}_t$ as a subgroup of $\mathrm{GL}_t$, we may restrict Schur modules to $\mathfrak{S}_t$ and decompose the result into a direct sum of Specht modules, the irreducible representations of $\mathfrak{S}_t$. We give an equivariant M\"{o}bius inversion formula that we use to invert this expansion in the representation ring for $\mathfrak{S}_t$ for $t$ large. In addition to explicit formulas in terms of plethysms, we show the coefficients that appear alternate in sign by degree. In particular, this allows us to define a new basis of symmetric functions whose structure constants are stable Kronecker coefficients and which expand with alternating signs into the Schur basis.
\end{abstract}

\maketitle

%
\section{Overview of main results}
%
\label{sec:introduction}

This paper concerns the relation between the representation theories of the general linear group $\GL_t$ and the symmetric group $\Sym_t$ over $\CC$. To fix notation, for $\lambda$ an integer partition, let $\ell(\lambda)$ denote the \emph{length} of $\lambda$ (number of nonzero parts), and let $|\lambda|$ denote the \emph{size} of $\lambda$ (sum of the parts). Let $\SS_{\lambda}$ denote the \newword{Schur functor}, so that the irreducible polynomial representations of $\GL_t$ are $\SS_{\lambda}(\CC^t)$ where $\ell(\lambda) \leq t$. For $\nu$ an integer partition, let $\Sp_{\nu}$ be the \newword{Specht module} over $\CC$, so that the irreducible representations of $\Sym_t$ are $\Sp_{\nu}$ for $|\nu|=t$.

Since $\Sym_t \subset \GL_t$, we can restrict the $\GL_t$ representation $\SS_{\lambda} (\CC^t)$ to $\Sym_t$ and decompose the result into Specht modules. For a partition $\nu=(\nu_1, \nu_2, \ldots, \nu_r)$, and $t \geq \nu_1 + |\nu|$, we define $\nu^{(t)}$ to be the partition $(t-|\nu|, \nu_1, \nu_2, \ldots, \nu_r)$ of $t$. Using this notation, we can write the aforementioned restriction as
\begin{equation}
  \Res_{\Sym_t}^{\GL_t} \SS_{\lambda}(\CC^t) \cong \bigoplus_{\nu} \Sp_{\nu^{(t)}}^{\oplus a_{\lambda}^{\nu}(t)} ,
  \label{e:restrict-mod}
\end{equation}
where $a_{\lambda}^{\nu}(t)$ are, by definition, the non-negative multiplicities that arise. In other words, in the representation ring $\Rep(\Sym_t)$ of $\Sym_t$, we have 
\begin{equation}
  [\SS_{\lambda} (\CC^t)] = \sum_{\nu} a^{\nu}_{\lambda}(t) [\Sp_{\nu^{(t)}}] .
  \label{e:restrict}
\end{equation}

A classical result of Littlewood \cite{Lit35} states that $a_{\lambda}^{\nu}(t)$ is independent of $t$ for $t$ sufficiently large. Therefore we may define coefficients $a_{\lambda}^{\nu}$ by
\begin{equation}
  a_{\lambda}^{\nu} = \lim_{t \rightarrow \infty} a_{\lambda}^{\nu}(t) .
  \label{e:restrict-stable}
\end{equation}
For $\lambda,\nu$ partitions with $|\lambda| \leq |\nu|$, Littlewood showed $a_{\lambda}^{\nu} = \delta_{\lambda,\nu}$. In particular, we may regard $a_{\lambda}^{\nu}$ as entries of an infinite upper uni-triangular matrix with rows and columns indexed by partitions. It is natural to invert this matrix to define coefficients $b_{\lambda}^{\nu}$ by
\begin{equation}
  \left[ b_{\lambda}^{\nu} \right] = \left[ a_{\lambda}^{\nu} \right]^{-1} .
  \label{e:invert-stable}
\end{equation}
While the $a_{\lambda}^{\nu}$ are non-negative, the $b_{\lambda}^{\nu}$ are, a priori, merely integers. Our main result is the following.

\begin{theorem} \label{thm:MT}
  With the above notation, we have $(-1)^{|\lambda| - |\nu|} b_{\lambda}^{\nu} \geq 0$. 
\end{theorem}

As we explain in Section~\ref{KronMotivation}, the $b_{\lambda}^{\nu}$ have recently become of interest as part of a strategy for computing stable Kronecker coefficients, so this basic result concerning their signs seems of importance.

\subsection{Plethystic formulas} \label{plethysm}

We can give a precise formula for $b_{\lambda}^{\nu}$ using the language of plethysm.
If $\psi:\GL_m\rightarrow\GL_n$ has character $g$ and $\phi:\GL_n\rightarrow\GL_p$ has character $f$, then $\phi \circ \psi:\GL_m\rightarrow\GL_p$ has character $f[g]$, the \newword{plethysm} of $f$ and $g$.
In terms of symmetric polynomials, if $g = \sum_{\alpha} g_{\alpha} x^{\alpha}$ is the monomial expansion, then $f[g]$ is $f(y_1,\ldots,y_t)$, where the $y_i$ are defined by the identity
\[ \prod \left( 1 + y_i q \right) = \prod_{\alpha} \left( 1 + x^{\alpha} q \right)^{g_{\alpha}} . \]
In other words, if the $g_{\alpha}$ are non-negative, then $x^{\alpha}$ occurs $g_{\alpha}$ times in the multiset $(y_1, \ldots, y_t)$. For more details on plethysm, see \cite{Lit40} and \cite[(I.8)]{Mac95}.

Littlewood \cite{Lit58} gave a formula for restriction from $\GL_t$ to $\Sym_t$ as the following plethysm
\begin{equation}
  a_{\lambda}^{\nu}(t) = \langle s_{\lambda}, s_{\nu^{(t)}}[1+h_1+h_2+\cdots] \rangle ,
  \label{e:littlewood}
\end{equation}
where $s_{\lambda}(x_1,\ldots,x_t) = \mathrm{char}(\SS_{\lambda}(\CC^t))$ is the \newword{Schur polynomial} corresponding to the irreducible character for $\GL_t$, $h_n = s_{(n)}$ is the complete homogeneous symmetric polynomial, and the inner product for characters, corresponding to the Hall inner product on symmetric polynomials, is determined by $\langle s_{\lambda} , s_{\mu} \rangle = \delta_{\lambda,\mu}$.

Define the \newword{Lyndon symmetric function} $L_m$ by
\begin{equation}
  L_m =\frac{1}{n}\sum_{d | m} \mu(d) p_d^{m/d},
  \label{e:lyndon}
\end{equation}
The Lyndon symmetric function is the character of the $\GL_t$ action on the degree $m$ part of the free Lie algebra on $\CC^t$ and is the Frobenius character of  $\left( \Ind_{C_m}^{\Sym_m} e^{2\pi i/m} \right)$ where $C_m$ is the cyclic subgroup of $\Sym_m$ generated by the $m$-cycle $(12\cdots m)$. 
Using $L_m$, we can give an explicit formula for $b_{\lambda}^{\nu}$:


\begin{theorem} \label{thm:MT-plethysm}
  For $\lambda$ and $\nu$ partitions, we have
  \begin{equation}
    b_{\lambda}^{\nu} = \sum_{\nu/\mu \ \text{vert. strip}} (-1)^{|\nu| - |\lambda|} \langle s_{\mu^{T}} , s_{\lambda^{T}} [L_1 + L_2 + L_3 + \cdots] \rangle ,
    \label{e:plethysm}
  \end{equation}
  where $\lambda^{T}$ denotes the transpose of the partition $\lambda$ and $L_m$ is the Lyndon symmetric function. In particular, $(-1)^{|\nu|-|\lambda|} b_{\lambda}^{\nu}$ is a nonnegative integer.
\end{theorem}
We remark that $L_1+L_2+L_3+\cdots$ can be viewed as the $\GL_t$ character of the free Lie algebra on $\CC^t$; it is not clear what significance this has.

Our proofs involve an intermediate $\Sym_t$-representation $M^t_{\mu}$ defined by 
\begin{equation}
  M^t_{\mu} = \Ind_{S_{|\mu|} \times S_{t-|\mu|}}^{S_t} \Sp_{\mu} \boxtimes \One_{t - |\mu|}
  \label{e:Mmu}
\end{equation}
where $\One_k$ is the trivial representation of $\Sym_k$.

Since $M^t_{\mu}$ is an $\Sym_t$ representation, it is a positive combination of the Specht modules $\Sp_{\nu^t}$. We will show that, in turn, $\SS_{\lambda}(\CC^t)$ is positive in the $M^t_{\mu}$.
We derive our result~\eqref{e:plethysm} by composing a formula for $\Sp_{\nu^t}$ in terms of the $M_{\mu}^t$ and a formula for $M_{\mu}^t$ in terms of $\SS_{\lambda}(\CC^t)$.
The following theorem gives plethystic formulas for transitioning between each of these bases.

\begin{theorem}\label{thm:MT FS}
In the representation ring $\Rep(\Sym_t)$, we have:
  \begin{eqnarray}
    [ M^t_{\mu} ] & = & \sum_{\nu} \langle s_{\nu} , s_{\mu}[1 + h_1] \rangle [\Sp_{\nu^t}] = \sum_{\nu/\mu \ \text{horiz. strip}}[\Sp_{\nu^t}]  . \label{e:M2Sp}\\ {} 
    [\Sp_{\nu^{(t)}}] & = & \sum_{\mu} \langle s_{\mu^{T}} , s_{\nu^{T}}[-1 + h_1] \rangle [M^t_{\mu}] = \sum_{\nu/\mu \ \text{vert. strip}} (-1)^{|\nu|-|\mu|} [M^t_{\mu}]  . \label{e:Sp2M}\\{}
    [ \SS^{\lambda} (\CC^t)] & = & \sum_{\mu} \langle s_{\lambda}, s_{\mu}[h_1+h_2+ h_3 + \cdots] \rangle [M_{\mu}^t]  . \label{e:S2M}\\{}
    [ M^t_{\mu} ] & = & \sum_{\lambda} (-1)^{|\mu| - |\lambda|} \langle s_{\mu^{T}} , s_{\lambda^{T}} [L_1 + L_2 + L_3 + \cdots ] \rangle [ \SS^{\lambda} (\CC^t)] .\label{e:M2S}
  \end{eqnarray}
\end{theorem}

The representation $M_{\mu}^t$ arises naturally in studying representations of the category of finite sets.  A representation of the category of finite sets consists of a sequence of vector spaces $V_0$, $V_1$, $V_2$, \dots and, for each map $\phi: \{ 1,2,\ldots, t \} \longrightarrow \{ 1,2,\ldots, u \}$ of finite sets, a map $\phi_{\ast} : V_t \to V_u$ obeying the obvious functoriality. In particular, each $V_t$ is a representation of the symmetric group $\Sym_t$. The category of such representations is an abelian category in an obvious manner. 
The simple objects in this category are explicitly described by Rains~\cite{Rains09} and are implicitly described in the work of Putcha~\cite{Putcha96}; see also Wiltshire-Gordon~\cite{WG14}. These simple objects $W_{\mu}$ are indexed by partitions and, except when $\mu$ is of the form $1^k$, we have $(W_{\mu})_t \cong M_{\mu}^t$ as an $\Sym_t$-representation. Wiltshire-Gordon also showed that the
$\SS_{\lambda} (\CC^t)$ are projective objects in this category.
 Thus, the problem of expanding $\SS_{\lambda} (\CC^t)$ positively in $M_{\mu}^t$ is the problem of finding the Jordan-Holder constituents of these projectives, and the problem of writing $M^t_{\mu}$ as an alternating combination of the $\SS_{\lambda} (\CC^t)$ is a combinatorial shadow of the problem of finding projective resolutions of these simples.

\subsection{Stable Kronecker coefficients} \label{KronMotivation}

The authors' original motivation for studying this problem came from a desire to understand tensor product multiplicities. The \newword{Kronecker coefficients}, denoted by $g_{\alpha,\beta,\gamma}$, indexed by a triple of partitions of the same size, give the Specht module decomposition of a tensor product of Specht modules, namely
\begin{displaymath}
  \Sp_{\alpha} \otimes \Sp_{\beta} \cong \bigoplus_{\gamma} \Sp_{\gamma}^{\oplus g_{\alpha, \beta, \gamma}} .
\end{displaymath}
Letting $\chi_{\alpha} = \mathrm{char}(\Sp_{\alpha})$, we can express these coefficients symmetrically as
\begin{displaymath}
  g_{\alpha, \beta, \gamma} = \langle \chi_{\alpha} \chi_{\beta} , \chi_{\gamma} \rangle
  = \frac{1}{t!} \sum_{w \in \Sym_t} \chi_{\alpha}(w) \chi_{\beta}(w) \chi_{\gamma}(w) ,
\end{displaymath}
where $t$ denotes the common size of $\alpha,\beta,\gamma$. It remains an important open problem in combinatorial representation theory to give a manifestly positive combinatorial formula for the Kronecker coefficients $g_{\alpha, \beta, \gamma}$.

For arbitrary partitions $\alpha,\beta,\gamma$ (potentially of varying sizes), Murnaghan \cite{Mur38} considered the coefficients $g_{\alpha^{(t)}, \beta^{(t)}, \gamma^{(t)}}$ and noticed that they stabilize for $t$ sufficiently large. This stability was proved by Brion \cite{Bri93}, and so we define the \newword{stable Kronecker coefficients}, denoted by $\gbar_{\alpha,\beta,\gamma}$, by
\begin{equation}
  \gbar_{\alpha,\beta,\gamma} =  \lim_{t \rightarrow \infty} g_{\alpha^{(t)}, \beta^{(t)}, \gamma^{(t)}} .
  \label{e:stable-kron}
\end{equation}
Giving a combinatorial rule for the stable Kronecker coefficients $\gbar_{\alpha,\beta,\gamma}$ is a major open problem in combinatorial representation theory.

By contrast with Kronecker coefficients and their stable variations, the tensor product coefficients for representations of $\GL_t$ are well understood. For partitions $\lambda,\mu,\nu$ with $|\lambda| + |\mu| = |\nu|$, define the \newword{Littlewood--Richardson coefficients}, denoted by $c_{\lambda,\mu}^{\nu}$, by
\begin{equation}
  \SS_{\lambda}(\CC^t) \otimes \SS_{\mu}(\CC^t) \cong \bigoplus_{\nu} \SS_{\nu}(\CC^t)^{\oplus c_{\lambda,\mu}^{\nu}} .
  \label{e:LRC}
\end{equation}
Taking characters, we may also define $c_{\lambda,\mu}^{\nu}$ by taking the Schur expansion of a product of Schur polynomials,
\begin{displaymath}
  s_{\lambda} s_{\mu} = \sum_{\nu} c_{\lambda,\mu}^{\nu} s_{\nu} .
\end{displaymath}
There are myriad combinatorial rules for $c_{\lambda,\mu}^{\nu}$, the first due to Littlewood and Richardson \cite{LR34} that was later proved by Sch{\"u}tzenberger \cite{Sch77} based on ideas of Robinson \cite{Rob38}. It is natural to try to exploit this understanding to study stable Kronecker coefficients. 

Indeed, based on the stable limit of \eqref{e:restrict}, we may define an inhomogeneous basis for symmetric polynomials, which we call \newword{stable Specht polynomials} and denote by $\sdag_{\nu}$, by the formula
\begin{equation}
  s_{\lambda} = \sum_{\nu} a_{\lambda}^{\nu} \sdag_{\nu} .
  \label{e:stable-s2d}
\end{equation}
Roughly, we are describing a map from the representation ring of $\Sym_t$ to symmetric functions sending $\SS_{\lambda}(\CC^t)$ to $s_{\lambda}$ and $\Sp_{\nu^{(t)}}$ to $s^{\dagger}_{\nu}$. (This statement is rough because we have not explained how to take the limit as $t \to \infty$ of representation rings.)
Note this map is very different from the Frobenius characteristic sending $\Sp_{\nu}$ to $s_{\nu}$.

Since restriction from $\GL_t$ to $\Sym_t$ restricts with tensor product, the structure constants of the stable Specht polynomials are stable Kronecker coefficients,
\begin{equation}
  \sdag_{\alpha} \sdag_{\beta} = \sum_{\gamma} \gbar_{\alpha, \beta, \gamma} \sdag_{\gamma} .
\end{equation}
Therefore a direct combinatorial description of stable Specht polynomials might well lead to a combinatorial rule for the stable Kronecker coefficients $\gbar_{\alpha, \beta, \gamma}$.

Schur polynomials are manifestly stable Specht-positive by \eqref{e:stable-s2d}. As one begins computing the $\sdag$ polynomials, one immediately notices they appear to be Schur-alternating. Theorem~\ref{thm:MT} proves this alternation.

\begin{corollary} \label{cor:alt-stable}
  The stable Specht polynomials are alternatingly Schur positive. Precisely, we have
  \[  \sdag_{\nu} = \sum_{\lambda} b_{\lambda}^{\nu} s_{\lambda} , \]
  where, in particular, $(-1)^{|\lambda| - |\nu|} b_{\lambda}^{\nu} \geq 0$.
\end{corollary}

Tables~\ref{tab:s2sdag} and \ref{tab:sdag2s} at the end of this paper give the expansions between Schur functions and stable Specht functions for degree up to $5$. 

This same basis of stable Specht polynomials has been discovered independently by Orellana and Zabrocki \cite{OZ16,OZ17}, who have been developing tools and techniques that might yet yield new insights into the stable Kronecker coefficients.

%
\section{Transition between Specht modules and $M_{\mu}^t$}
%
\label{sec:SpechtM}

In this section, we carry out the relatively easy task of relating the classes of $\Sp_{\nu^{(t)}}$ and $M_{\mu}^t$ in the representation ring $\Rep(\Sym_t)$. 

\begin{proposition}
  In the representation ring $\Rep(\Sym_t)$, we have
  \begin{eqnarray}
    [ M^t_{\mu} ] & = & \sum_{\mu/\nu \ \text{horiz strip}} \hspace{-1ex} [\Sp_{\nu^{(t)}}] \label{e:M2SpStrip}
    \\ {} 
    [\Sp_{\nu^{(t)}}] & = & \sum_{\nu/\mu \ \text{vert  strip}} \hspace{-1ex} (-1)^{|\nu| - |\mu|} [M^t_{\mu}]. \label{e:Sp2MStrip}
  \end{eqnarray}
  \label{prop:Sp2M}
\end{proposition}

\begin{proof}
  By Pieri's rule \cite{Pie93} for induction, from Eq.~\eqref{e:Mmu} we immediately have
  \begin{equation}
    M^t_{\mu} \cong \bigoplus_{|\lambda| =t,\ \lambda/\mu \ \text{horiz. strip}} \Sp_{\lambda} . 
    \label{e:Pieri_induct}
  \end{equation}
  For a partition $\lambda$ with parts $\lambda_1 \geq \lambda_2 \geq \cdots \geq \lambda_{\ell}$, we let $\bar{\lambda}$ be the partition $(\lambda_2, \ldots, \lambda_{\ell})$, so $\lambda = \bar{\lambda}^{(t)}$ if $t = |\lambda|$. 
  We note that $\lambda/\mu$ is a horizontal strip if and only if $\lambda_1 \geq \mu_1 \geq \lambda_2 \geq \mu_2 \geq \cdots$. Holding $\bar{\lambda}$ and $\mu$ fixed, once $t$ is sufficiently large, the condition that $\lambda_1 \geq \mu_1$ is automatic, so $\lambda/\mu$ is a horizontal strip if and only if $\mu_1 \geq \lambda_2 \geq \mu_2 \geq \cdots$; the latter states that $\mu/\bar{\lambda}$ is a horizontal strip. 
  So, for $t$ sufficiently large, we have
\[ 
    M^t_{\mu} \cong \bigoplus_{\mu/\bar{\lambda} \ \text{horiz. strip}} \Sp_{\bar{\lambda}^{(t)}} . 
\]
Renaming the summation variable $\bar{\lambda}$ as $\nu$, we have proved Eq.~\eqref{e:M2SpStrip}.

  To arrive at the second formula, we must invert the infinite $0-1$ matrix with entries $M_{\mu \nu} = 1$ if and only if $\mu/\nu$ is a horizontal strip. By Pieri's rule, $h_r s_{\nu}$ is the sum $\sum_{\mu/\nu \ \text{horiz. $r$-strip}} s_{\mu}$.
  Thus, as an endomorphism of the completion of the ring of symmetric functions, the matrix $M$ corresponds to multiplication in the Schur basis by $1 + h_1 + h_2 + h_3 \cdots$. The inverse operation is multiplication by $(1+h_1+h_2 + \cdots )^{-1} = 1 - e_1 + e_2 - e_3 + \cdots$. By Pieri's rule, $e_r s_{\mu}$ is the sum $\sum_{\nu/\mu \ \text{vert. $r$-strip}} s_{\nu}$. So the inverse matrix has $(M^{-1})_{\mu \nu}$ equal to $(-1)^r$ if $\nu/\mu$ is a vertical strip of size $r$, and $0$ otherwise.  Eq.~\eqref{e:Sp2MStrip} follows.
\end{proof}

In particular, notice that the transition coefficients between $[M_{\mu}^t]$ and $[\Sp_{\lambda}]$ are independent of $t$. 
Recalling that $s_{\lambda}[1+h_1] = \sum_{\lambda/\mu \ \text{horiz. strip}} s_{\mu}$, we deduce:

\begin{corollary}
  In the representation ring $\Rep(\Sym_t)$, we have 
  \begin{eqnarray}
    [ M^t_{\mu} ] & = & \sum_{\nu} \langle s_{\nu} , s_{\mu}[1 + h_1] \rangle [\Sp_{\nu^{(t)}}] , \label{e:M2Sp-2} \\ {} 
    [\Sp_{\nu^{(t)}}] & = & \sum_{\mu} \langle s_{\mu^{T}} , s_{\nu^{T}}[-1 + h_1] \rangle [M^t_{\mu}], \label{e:Sp2M-2}
  \end{eqnarray}
  where $\lambda^{T}$ denotes the transpose of $\lambda$.
\end{corollary}

This proves the first two parts of Theorem~\ref{thm:MT FS}.
%

\section{Background on wreath products}
\label{sec:wreath}



We recall that, for two groups $G$ and $H$, the \newword{wreath product} $G \wreath H$ is the semidirect product $\left(\prod_{h \in H} G \right) \rtimes H$, where $H$ acts by permuting the $G$-factors. The wreath product $\Sym_j \wreath \Sym_m$ embeds in $\Sym_{jm}$ as the normalizer of the Young subgroup $\Sym_j^{\times m}$, and we'll write $\Wr(j,m)$ for this subgroup of $\Sym_{jm}$.

For $V$ a representation of $\Sym_j$ and $W$ a representation of $\Sym_m$, we write $V \S W$ for $V^{\otimes m} \otimes W$ as a representation of $\Wr(j,m)$ where $\Wr(j,m)$ acts in the obvious way on $V^{\otimes m}$ and through the quotient $\Sym_m$ on $W$. 
We recall:

\begin{theorem}[{\cite[Theorem A2.8]{EC2}}] \label{WreathInductionPlethysm}
With notation as above, we have
  \begin{equation}
    \ch\left( \Ind_{\Wr(j,m)}^{\Sym_{ab}} (V \emph{\S} W) \right) = \ch(W)[\ch(V)] ,
  \end{equation}
  where $\ch$ denotes the Frobenius characteristic map.
  \label{thm:Stanley-Wr}
\end{theorem}

As the special case where $V$ is the trivial representation, we have
\begin{corollary} \label{SpecialCaseWreathInduction}
Let $\mu$ be a partition of $m$; we can consider $\Sp_{\mu}$ as a representation of $\Wr(j,m)$ through the quotient $\Wr(j,m) \to \Sym_m$. Then
\[ \ch \left( \Ind_{\Wr(j,m)}^{\Sym_{jm}} \Sp_{\mu} \right) = s_{\mu}[h_j] . \] 
\end{corollary}

We also want to embed the wreath product into $\Sym_{jm} \times \Sym_m$. Let $V(j,m) \subset \Wr(j,m) \times \Sym_m$ be the graph of the map $\Wr(j,m) \cong \Sym_j^{\times m} \rtimes \Sym_m \to \Sym_m$, so $V(j,m)$ is a subgroup of $\Sym_{jm} \times \Sym_m$ isomorphic to $\Sym_j \wreath \Sym_m$. 

\begin{lemma} \label{key wreath induction}
We have the following equality in $\Rep(\Sym_{jm} \times \Sym_m)$:
\[  \left[ \Ind_{V(j,m)}^{\Sym_{jm} \times \Sym_m}  \One  \right] \ \ = \  \bigoplus_{\substack{ |\lambda| = jm \\ |\mu| = m}} \langle s_{\lambda}, s_{\mu}[h_j] \rangle\ \left[\Sp_{\lambda} \boxtimes \Sp_{\mu} \right].  \]
\end{lemma}

\begin{proof}
  We first compute the induction from $V(j,m)$ to $\Wr(j,m) \times \Sym_m$, and then further induct to $\Sym_{jm} \times \Sym_m$. 
  Let $W$ be the representation $\Ind_{V(j,m)}^{\Wr(j,m) \times \Sym_m} \One$. We note that $W$ factors through the quotient $\Sym_m \times \Sym_m$ of $\Wr(j,m) \times \Sym_m$. As such, we have $W \cong \CC \Sym_m$ with the two actions of $\Sym_m$ coming from the left and right actions of $\Sym_m$ on itself. 
  By the Peter-Weyl theorem (and using that all representations of $\Sym_m$ are self dual), we have
  \[ \Ind_{V(j,m)}^{\Wr(j,m) \times \Sym_m} \One \cong \bigoplus_{|\mu| = m} \Sp_{\mu} \boxtimes \Sp_{\mu} \]
  where the action of $\Wr(j,m)$ is through its quotient $\Sym_m$. 
  
  We now induce to $\Sym_{jm}$:
  \[ \Ind^{\Sym_{jm} \times \Sym_m}_{\Wr(j,m) \times \Sym_m} \Ind_{V(j,m)}^{\Wr(j,m) \times \Sym_m} \One = \bigoplus_{|\mu| = m} \left( \Ind_{\Wr(j,m)}^{\Sym_{jm}} \Sp_{\mu} \right) \boxtimes \Sp_{\mu} . \]
  The inner induction can be computed by Corollary~\ref{SpecialCaseWreathInduction}:
  \[ \Ind_{\Wr(j,m)}^{\Sym_{jm}} \Sp_{\mu} \cong \bigoplus_{|\lambda| = jm}  \Sp_{\lambda}^{\oplus\langle s_{\lambda}, s_{\mu}[h_j] \rangle}  . \]
  Putting all of our formulas together, we deduce the result.
\end{proof}

%

\section{Expansion of $\SS_{\lambda}(\CC^t)$ in the basis $M_{\mu}^t$} 

Our next task is to give a formula for the expansion of $\SS_{\lambda}(\CC^t)$ in terms of the representation $M_{\mu}^t$. This result is of some interest in itself, but is more important as a preview of the methods we will use to express $M_{\mu}^t$ in terms of $\SS_{\lambda}(\CC^t)$. 

We write $T(l,t)$ for $(\CC^t)^{\otimes l}$, which is an $\Sym_{l} \times \Sym_t$ representation in the obvious manner. 
The obvious basis for $T(l,t)$ is $\mathbf{e}_{i_1} \otimes \mathbf{e}_{i_2} \otimes \cdots \otimes \mathbf{e}_{i_{l}}$, where $i_1 i_2 \cdots i_{l}$ runs over all length $l$ words in the alphabet $\{ 1,2, \ldots, t \}$; we abbreviate this basis element $[i_1 i_2 \cdots i_{l}]$.
We write $D(l,t)$ for the subspace of $T(l,t)$ with basis those $[i_1 i_2 \cdots i_{l}]$ where $i_1 i_2 \cdots i_{l}$ are pairwise distinct; this is clearly a $\Sym_{l} \times \Sym_t$ subrepresentation.

Let $\Pi_{l}$ denote the lattice of set partitions of $\{ 1,2 \ldots, l \}$ ordered by refinement, with minimal element $\Fine_{l} := \{ \{ 1 \}, \{ 2 \}, \ldots, \{ l \} \}$ and maximal element $\Coarse_{l} := \{ \{ 1,2,\ldots, l \} \}$. For a set partition $\pi = \{ \pi_1, \pi_2, \ldots, \pi_{m} \}$, we write $\Shape(\pi)$ for the partition obtained by sorting $(|\pi_1|, |\pi_2|, \dots, |\pi_{m}| )$ into order. We abbreviate $|\Shape(\pi)|$ and $\ell(\Shape(\pi))$ to $|\pi|$ and $\ell(\pi)$. 
In order to help the reader distinguish integer partitions from set partitions, we will consistently denote the former by the letters $\lambda$, $\mu$, $\nu$ and the latter by $\pi$, $\rho$, $\sigma$.

Given a set partition $\pi$ of $\{ 1,2 \ldots, l \}$, we write
\[ \begin{array}{lcr}
  T(\pi,t) & = & \left\{ [i_1 i_2 \cdots i_{l}] \in T(l,t) \mid i_p = i_q \mbox{ if } p,q \in \pi_j \mbox{ for some } j \right\} , \\
  D(\pi,t) & = & \left\{ [i_1 i_2 \cdots i_{l}] \in T(l,t) \mid i_p = i_q \mbox{ if and only if } p,q \in \pi_j \mbox{ for some } j \right\} .
\end{array} \]
In particular, we have $T(\Fine_{l},t) = T(l,t)$ and $D(\Fine_{l},t) = D(l,t)$, and also $D(\Coarse_{l},t) = T(\Coarse_{l},t) \cong \mathbb{C}^t$. We may relate these two constructions by
\begin{eqnarray}
  T(\pi,t) & = & \bigoplus_{\rho \succeq \pi} D(\rho, t) .
\end{eqnarray}

We will also consider representations indexed by integer partitions, rather than by set partitions. For a partition $\nu$ of ${l}$, we set
\[ \DSh(\nu,t) = \bigoplus_{\Shape(\pi) = \nu} D(\pi, t) .\] 
We will now compute the character of $\DSh(\nu, t)$.

\begin{lemma}\label{lem:Dkt}
  For $\nu$ a partition of $l$ with length $m$, and for $t \geq m$, in the representation ring for $\Sym_{l} \times \Sym_{t}$, we have
  \begin{equation}
    [ \DSh(\nu,t) ] = \sum_{\substack{|\lambda| = l \\ |\mu| = m}} \sum_{\substack{|\mu(j)| = m_j}} c_{\mu(1) \mu(2) \cdots \mu(r)}^{\mu} \langle s_{\lambda}, \prod_j s_{\mu(j)}[h_j] \rangle \left[ \Sp_{\lambda} \boxtimes M_{\mu}^t \right] .
    \label{e:Dkt}
  \end{equation} 
  where $\nu = r^{m_r} \cdots 2^{m_2} 1^{m_1}$ and each $\mu(j)$ is a partition of size $m_j$. 
\end{lemma}

\begin{proof}
  Notice $\DSh(\nu,t)$ is a permutation representation with basis $\{ g \cdot \mathbf{x}_{\nu} \}_{g \in \Sym_{l} \times \Sym_t}$ where 
  \[  \mathbf{x}_{\nu} = [ \overbrace{11 \cdots 1}^{\nu_1} \overbrace{22 \cdots 2}^{\nu_2} \cdots \overbrace{m m \cdots m}^{\nu_{m}} ] \in \DSh(\nu, t) .\]

  We first consider the case $t=m$. We have $\sum_j j m_j = l$ and $\sum_j m_j = m$, and so the stabilizer of $\mathbf{x}_{\nu}$ in $\Sym_{l} \times \Sym_{m}$ is $\prod V(j, m_j) \subseteq \prod (\Sym_{j m_j} \times \Sym_{m_j}) \subseteq \Sym_{l} \times \Sym_{m}$. Therefore
  \[ \DSh(\nu, m) = \Ind_{\prod V(j, m_j) }^{\Sym_{l} \times \Sym_{m}} \One .\]

  We perform the induction in two steps. First, by Lemma~\ref{key wreath induction}, we have
  \[
  \left[ \Ind_{\prod V(j, m_j) }^{\prod (\Sym_{j m_j} \times \Sym_{m_j}) } \One \right]  = \prod_j  \sum_{\substack{ |\lambda(j)| = j m_j \\ |\mu(j)| = m_j}}  \langle s_{\lambda(j)}, s_{\mu(j)}[h_j] \rangle \   \left[ \Sp_{\lambda(j)} \boxtimes \Sp_{\mu(j)} \right]
  \]
  in $\Rep(\prod (\Sym_{j m_j} \times \Sym_{m_j}))$. We emphasize that each of the $\lambda(j)$ is a partition, they are not the parts of a partition named $\lambda$, and likewise for the $\mu(j)$'s. Interchanging summation and product, we get
  \[
  \left[ \Ind_{\prod V(j, m_j) }^{\prod (\Sym_{j m_j} \times \Sym_{m_j}) } \One \right]  = \sum_{\substack{ |\lambda(j)| = j m_j \\ |\mu(j)| = m_j}}   \prod_j  \langle s_{\lambda(j)}, s_{\mu(j)}[h_j] \rangle \ \left[\Sp_{\lambda(j)} \boxtimes \Sp_{\mu(j)} \right].
  \]
  Inducing further, using the classical result
 $\Ind_{\Sym_{l} \times \Sym_m}^{\Sym_{m+l}} [\Sp_{\lambda} \boxtimes \Sp_{\mu}] = \sum_{\nu} c_{\lambda \mu}^{\nu} [\Sp_{\nu}]$, gives
 \begin{multline}
    \left[ \Ind_{\prod V(j, m_j) }^{\Sym_{l} \times \Sym_{m}  } \One \right] = \\
    \sum_{\substack{|\lambda| = l \\ |\mu| = m}} \sum_{\substack{ |\lambda(j)| = j m_j \\ |\mu(j)| = m_j}} \prod_j  \langle s_{\lambda(j)}, s_{\mu(j)}[h_j] \rangle \  c_{\lambda(1) \lambda(2)  \cdots \lambda(r)}^{\lambda} \ c_{\mu(1)  \mu(2)  \cdots \mu(r)}^{\mu} \left[ \Sp_{\lambda} \boxtimes \Sp_{\mu} \right] .
    \label{e:big_sum}
  \end{multline}

  To simplify this, note that for symmetric functions $f,g$ homogeneous of degrees $a, b$, respectively, we have
  $\langle s_{\lambda} , f g \rangle = \sum_{|\alpha| = a , |\beta|=b} c_{\alpha,\beta}^{\lambda} \langle s_{\alpha} , f \rangle \langle s_{\beta} , g \rangle$. Thus we may use the coefficients $c_{\lambda(1) \lambda(2) \cdots \lambda(r)}^{\lambda}$ to reduce Eq.~\eqref{e:big_sum} to
  \begin{equation}
    \left[ \Ind_{\prod V(j, m_j) }^{\Sym_{l} \times \Sym_{m}  } \One \right] =
    \sum_{\substack{|\lambda| = l \\ |\mu| = m}} \sum_{\substack{|\mu(j)| = m_j}} \langle s_{\lambda}, \prod_j s_{\mu(j)}[h_j] \rangle \ c_{\mu(1) \mu(2) \cdots \mu(r)}^{\mu} \left[ \Sp_{\lambda} \boxtimes \Sp_{\mu} \right] .
    \label{e:med_sum}
  \end{equation}
  
  Finally, inducing from $\Sym_{m} \times \Sym_{t- m}$ to $\Sym_t$ gives Eq.~\eqref{e:Dkt}.
\end{proof}

We now establish the third part of Theorem~\ref{thm:MT FS}.

\begin{proposition}
  In the representation ring $\Rep(S_t)$, we have
  \begin{equation}
    [ \SS^{\lambda} (\CC^t)] = \sum_{\mu} \langle s_{\lambda}, s_{\mu}[h_1+h_2+ h_3 + \cdots] \rangle [M_{\mu}^t] . 
  \end{equation}
  \label{prop:M2S}
\end{proposition}

\begin{proof}
  We begin with the decomposition
  \begin{equation}
  T(l,t) = \bigoplus_{|\nu| = l} \DSh(\nu, t) . \label{T decomposition}
  \end{equation}
  Both sides of this equation have compatible actions of $\Sym_{l} \times \Sym_t$. 
  On the left hand side, Schur-Weyl duality tells us that
  \begin{equation}
  T(l,t) = (\CC^t)^{\otimes l} \cong \bigoplus_{|\lambda|=l} \Sp_{\lambda} \boxtimes \SS_{\lambda}(\CC^t) . \label{Schur Weyl}
  \end{equation}
  So we can compute $\SS_{\lambda}(\CC^t)$ as the $\Sp_{\lambda}$-component of the right hand side of Eq.~\eqref{T decomposition}. 

From Eq.~\eqref{e:Dkt}, the coefficient of $\Sp_{\lambda}$ in $\DSh(\nu,t)$ is
\[
\sum_{\substack{ |\mu| = m \\ |\mu_{(j)}| = m_j}} c_{\mu_{(1)} \mu_{(2)} \cdots \mu_{(r)}}^{\mu} \langle s_{\lambda}, \prod_j s_{\mu_{(j)}}[h_j] \rangle \left[ M_{\mu}^t \right]
\]
where $\nu = r^{m_r} \cdots 2^{m_2} 1^{m_1}$. Thus the coefficient of $\Sp_{\lambda}$ in $T(l,t)$ is obtained by summing over $\nu$, which gives 
\[
\sum_{\substack{ |\mu| = m \\ \sum |\mu_{(j)}| = m}} \langle s_{\lambda}, c_{\mu_{(1)} \mu_{(2)} \cdots \mu_{(r)}}^{\mu} \prod_j s_{\mu_{(j)}}[h_j] \rangle \left[ M_{\mu}^t \right]  = \sum_{|\mu| = m} \langle s_{\lambda}, s_{\mu}[\sum h_j] \rangle \left[ M_{\mu}^t \right],
\]
where we use the identity $s_{\mu}[f+g] = \sum_{\alpha,\beta} c_{\alpha,\beta}^{\mu} s_{\alpha}[f] s_{\beta}[g]$ (see \cite[I(8.8)]{Mac95}). 
\end{proof}

Combining Propositions~\ref{prop:Sp2M} and \ref{prop:M2S}, we recover Littlewood's formula for $a_{\lambda}^{\nu}$.

\begin{corollary}
  In the representation ring $\Rep(S_t)$, we have
  \begin{eqnarray*}
    [ \SS^{\lambda} (\CC^t)] = \sum_{\nu} \langle s_{\lambda}, s_{\nu^{(t)}}[1 + h_1 + h_2 + \cdots] \rangle [\Sp_{\nu^{(t)}}] . 
  \end{eqnarray*}
\end{corollary}

\section{Equivariant M\"obius inversion}

Our proof of the third part of Theorem~\ref{thm:MT FS} began with the identity $T(l,t) = \bigoplus_{|\nu| = l} \DSh(\nu, t)$. 
To establish the fourth, and most interesting, part we must invert this expression and write $D(l, t)$ as a ``linear combination" of the representations $T(\pi(\nu), t)$, where $\pi(\nu)$ is a set partition of shape $\nu$. 
We can use M\"obius inversion on the set partition lattice to compute the dimension of $D(l,t)$ as a linear combination of the dimensions of the representations $T(\pi, t)$. 
In order to obtain not just the dimension, but a formula for the class in $\Rep(\Sym_{l}\times \Sym_t)$, we need an equivariant version of M\"obius inversion. 
We find it clearest to explain this result in the context of a general poset with a group action.
Because we want to reserve $\mu$ for partitions, we will denote the M\"obius function of a poset by $\mmu$.

Let $P$ be a poset with unique minimal element $\ohat$, and let $G$ be a group acting on $P$. 
Let $V$ be a $G$-representation with a direct sum decomposition $V = \bigoplus_{p \in P} U_p$ such that $g(U_p) = U_{gp}$ for each $g \in G$ and $p \in P$. 
For $q \in P$, put $V_q := \bigoplus_{r \succeq q} U_r$.

For $p \in P$, let $(\ohat,p) = \{ q \in P  : \ohat \prec q \prec p \}$. 
Let $\Delta(\ohat,p)$ be the order complex of $(\ohat,p)$ -- the simplicial complex on the ground set $(\ohat,p)$ whose faces are the totally ordered subsets of $(\ohat,p)$. 
A classical result of P. Hall states \cite{Hal88} that the M\"obius function $\mmu(p)$ is the reduced Euler characteristic $\tilde{\chi}(\Delta(\ohat,p))$.

We will define an equivariant version of $\mmu$.
Namely, let $G_p$ be the stabilizer of $p$ and let $\Rep(G_p)$ be its representation ring. 
We define
\[ \meq(p) = \sum_j (-1)^{j+1} \left[ \widetilde{H}_j(\Delta(\ohat,p)) \right] \]
where $\widetilde{H}_j$ is the reduced homology group. 
So, under the map $\Rep(G_p) \to \mathbb{Z}$ sending a representation to its dimension, $\meq(p)$ is sent to the M\"obius function $\mmu(p)$. 
Among group theorists, $\meq(p)$ is called the ``Lefschetz element".

Let $G \backslash P$ be a set of orbit representatives for the action of $G$ on $P$.
Our equivariant M\"obius inversion formula is the following.

\begin{theorem}
  With the above definitions, we have the equality
  \begin{equation} [U_{\ohat}] = \sum_{p \in G \backslash P} \left[ \Ind_{G_p}^G \left(\meq(p) \otimes V_p\right) \right] \label{eqn:mobius} \end{equation}
  in the representation ring $\Rep(G)$.
  \label{thm:mobius}
\end{theorem}
%
%
%
%

\begin{proof}
  We begin by expanding the induction. For a simplicial complex $X$, let $C_j(X)$ be the free vector space on the $j$-dimensional faces. Here we include the case $j=-1$, corresponding to the empty face. If a group $H$ acts on $X$,  we have $[\meq(X)] = \sum_{j \geq -1} (-1)^{j+1} [C_j(X)]$ in $\Rep(H)$, so the right hand side of~\eqref{eqn:mobius} is
  \[ \sum_{j \geq -1} (-1)^{j+1} \sum_{p \in G \backslash P} \left[  \Ind_{G_p}^G \ C_j(\Delta(\ohat,p)) \otimes  V_p   \right] . \]

  The induction can be expanded explicitly as 
  \[\sum_{p \in G \backslash P} \left[ \bigoplus_{p' \in Gp} C_j(\Delta(\ohat,p')) \otimes V_{p'}   \right] . \]
  Summing over $p \in G \backslash P$ and $p' \in Gp$ is simply summing over $p' \in P$. So we want to prove the equality
  \[ [U_{\ohat}] =  \sum_{j \geq -1} (-1)^{j+1} \left[ \bigoplus_{p' \in P} V_{p'} \otimes C_j(\Delta(\ohat,p')) \right] =  \sum_{j \geq -1} (-1)^{j+1} \left[ \bigoplus_{\ohat \prec q_0 \prec q_1 \prec \cdots \prec q_j \prec p' \in P} V_{p'}  \right] \]
in $\Rep(G)$. Inserting the definition of $V_{p'}$,  our goal is to show that
  \[ [U_{\ohat}] =  \sum_{j \geq -1} (-1)^{j+1} \left[ \bigoplus_{\ohat \prec q_0 \prec q_1 \prec \cdots \prec q_j \prec p' \preceq q'} U_{q'}  \right] . \]
  Here $G$ acts by permuting the summation indices and by its action on $V$. 
  
  In order to show equality in $\Rep(G)$, we simply need to compute characters of both sides. Fix $g \in G$; for any vector space $W$ on which $g$ acts, write $\Tr_{g}(W)$ for the trace of $g$ on $W$. Since the action of $g$ on $P$ is order preserving, if $g$ maps a $j$-cell $(q_0, q_1, \ldots, q_j)$ of $\Delta(\ohat,p)$ to itself, it does so while preserving the order of $(q_0, q_1, \ldots, q_j)$. So the only terms that contribute to the trace are those where each of the summation variables $q_0$, $q_1$, \dots, $q_j$, $p'$ and $q'$ are individually fixed by $g$. 
  We obtain that the trace of $g$ on the right hand side is
  \[  \sum_{j \geq -1} (-1)^{j+1} \hspace{-1ex} \sum_{\substack{ q_0, q_1, \ldots, q_j, p', q' \in P^g \\ \ohat \prec q_0 \prec q_1 \prec \cdots \prec q_j \prec p' \preceq q'}} \hspace{-0.5ex} \Tr_{g}(U_{q'}) =
  \sum_{q' \in P^g} \Tr_{g}(U_{q'}) \big( \sum_{j \geq -1} \hspace{-1ex} \sum_{\substack{ q_0, q_1, \ldots, q_j \in P^g \\ \ohat \prec q_0 \prec q_1 \prec \cdots \prec q_j \prec p' \preceq q'  }} \hspace{-0.5ex} (-1)^{j+1} \big) . \]
   The quantity in parentheses in the reduced Euler characteristic of the order complex $\Delta\left( \{ q \in P^g : \ohat \prec q \preceq q' \} \right)$.
  For $q' \neq \ohat$, the point $q'$ is a cone point of the order complex, so this Euler characteristic is $0$.
  Thus, the sum simplifies to $\Tr_{g}(U_{\ohat})$, as desired.
\end{proof}

Our immediate purpose is to apply Theorem~\ref{thm:mobius} to the partition lattice $\Pi_m$, ordered by refinement with minimal element $\Fine_m$ and ranked by $m$ minus the number of blocks of the set partition. The group $\Sym_{m}$ acts by permuting elements within blocks. We can identify $\Sym_m \backslash \Pi_m$ with the set of integer partitions of $m$: for each integer partition $\nu$ of $m$, choose a set partition $\pi(\nu)$ of that shape. We will abbreviate the stabilizer $G_{\pi(\nu)}$ to simply $G_{\nu}$ and the equivariant M\"obius function $\meq(\pi(\nu))$ to simply $\meq(\nu)$. We may extend this action to an $\Sym_m \times \Sym_t$ action where the second factor acts trivially, and in so doing the corresponding objects for the action of $\Sym_m \times \Sym_t$ become $G_{\nu} \times \Sym_t$ and $\meq(\nu) \boxtimes \One$, respectively.

Recall the minimal element of $\Pi_m$ is $\Fine_{m}$, and $D(\Fine_m,t) = D(m, t)$. Applying Theorem~\ref{thm:mobius} to the representation $T(\pi,t) = \bigoplus_{\rho \succeq \pi} D(\rho,t)$ gives the following.

\begin{corollary}
  In the representation ring $\Rep(\Sym_{m}\times\Sym_{t})$, we have
  \begin{equation}
    [D(m,t)] = \sum_{|\nu|=m} \left[ \Ind_{G_{\nu}\times\Sym_{t}}^{\Sym_{m}\times\Sym_{t}} \left((\meq(\nu)\boxtimes\One) \otimes T(\pi(\nu),t) \right) \right] 
    \label{Mobius for D}
  \end{equation}
  where $\pi(\nu)$ is a set partition of shape $\nu$.
  \label{cor:Mobius for D}
\end{corollary}

\section{Expansion of $M_{\mu}^t$ in the basis $\SS_{\lambda}(\CC^t)$} 

We now prove the last part of Theorem~\ref{thm:MT FS}. Following the proof paradigm for the expansion of $\SS_{\lambda}(\CC^t)$ in the basis $M_{\mu}^t$, we will express the $\Sym_m \times \Sym_t$ representation $D(m, t)$ in two ways. Recall that $D(m, t) = D(\Fine_m,t) = \DSh((1^m),t)$. By Lemma~\ref{lem:Dkt}, we have the following.

\begin{proposition}
  As an $\Sym_m \times \Sym_t$ representation, we have
  \begin{equation}
    D(m,t) \cong \bigoplus_{|\mu| = m} \Sp_{\mu} \boxtimes M_{\mu}^t .
    \label{Schur Weyl for D}
  \end{equation}
\end{proposition}

\begin{proof}
  For $\nu = (1^m)$, Lemma~\ref{lem:Dkt} gives
  \[ [ \DSh((1^m),t) ] = \sum_{|\lambda| = |\mu| = m} \sum_{\substack{|\mu_{(1)}| = m}} c_{\mu_{(1)}}^{\mu} \langle s_{\lambda}, s_{\mu_{(1)}}[h_1] \rangle \left[ \Sp_{\lambda} \boxtimes M_{\mu}^t \right] . \]
  Using $c_{\mu_{(1)}}^{\mu} = \delta_{\mu,\mu_{(1)}}$ and $s_{\mu_{(1)}}[h_1] = s_{\mu_{(1)}}$ and $\langle s_{\lambda}, s_{\mu} \rangle = \delta_{\lambda,\mu}$ gives the result.
\end{proof}
  

In particular, $M_{\mu}^t$ can be identified with the $\Sp_{\mu}$-isotypic component of $D(m,t)$. On the other hand, equivariant M\"obius inversion gives the following.

\begin{lemma}
  In the representation ring $\Rep(\Sym_{m}\times\Sym_{t})$, we have
  \begin{equation}
    [D(m,t)] = \sum_{|\nu|=m} \sum_{|\lambda| = \ell(\nu)}  \left[ \Ind_{G_{\nu}}^{\Sym_m} (\meq(\nu) \otimes \Sp_{\lambda}) \boxtimes \SS^{\lambda}(\CC^t) \right] .
    \label{Distribute Mobius D}
  \end{equation}
  where $\pi(\nu)$ is a set partition of shape $\nu$.
\end{lemma}

\begin{proof}
  For $\nu$ a partition of $m$ with length $\ell$ and $\pi$ a set partition of shape $\nu$, as an $\Sym_t$ module, $T(\pi,t)$ is $(\CC^t)^{\otimes \ell}$. By Schur-Weyl duality, as an $\Sym_{\ell}\times\Sym_t$ module, this becomes $\bigoplus_{|\lambda| = \ell} \Sp_{\lambda} \boxtimes \SS^{\lambda}(\CC^t)$. Writing $\nu = 1^{m_1} 2^{m_2} \cdots r^{m_r}$, with notation as in Corollary~\ref{cor:Mobius for D}, the stabilizer of $\nu$ is $G_{\nu} = \prod \Wr(j, m_j) \subset \prod \Sym_{j m_j} \subset \Sym_m$. The group $G_{\nu}$ acts on the set of blocks of $\pi(\nu)$, giving a map $G_{\nu} \to \Sym_{\ell}$. Combining this with the action of $\Sym_{\ell}$ on $\Sp_{\lambda}$ turns Eq.~\eqref{Mobius for D} into Eq.~\eqref{Distribute Mobius D}, as desired.
\end{proof}
  
At this point, we can prove $[M_{\mu}^t]$ is alternating in the $[\SS_{\lambda}(\CC^t)]$.
By~\eqref{Schur Weyl for D}, $M_{\mu}^t$ can be identified with the $\Sp_{\mu}$-isotypic component of $D(m,t)$, where $m = |\mu|$. 
Combining this observation with~\eqref{Distribute Mobius D}, the coefficient of $[\SS_{\lambda}(\CC^t)]$ in $[M_{\mu}^t]$ is the coefficient of $[\Sp_{\mu}]$ in $\sum_{\nu} \Ind_{G_{\nu}}^{\Sym_m} (\meq(\nu) \otimes [\Sp_{\lambda}])$. For $\lambda$ and $\mu$ fixed, the only terms that contribute to this coefficient are partitions $\nu$ with $|\nu| = |\mu|$ and $\ell(\nu) = |\lambda|$. But $|\mu| - |\lambda| = |\nu| - \ell(\nu)$ is precisely the rank function that grades the lattice of set partitions. So only terms at one fixed level of $\Pi_m$ will contribute. Since $\Pi_m$ is Cohen-Macaulay \cite{Bjo80}, the terms $\meq(\nu)$ will all come with the same sign $(-1)^{|\mu| - |\lambda|}$, proving the multiplicity of $[\SS_{\lambda}(\CC^t)]$ in $[M_{\mu}^t]$ has sign $(-1)^{|\mu| - |\lambda|}$, as promised. 

Our final task is to establish our plethystic formula for the coefficient of $[\SS_{\lambda}(\CC^t)]$ in $[M_{\mu}^t]$ and thus in $[\Sp_{\nu^{(t)}}]$ in $\Rep(\Sym_t)$. For this, we must understand the equivariant M\"obius function for intervals in the partition lattice $\Pi_m$.

Bj\"orner \cite{Bjo80} showed that the order complex of $\Pi_m$ is Cohen--Macaulay, and so has reduced homology only in the top dimension. For $\pi$ a set partition of shape $\nu = 1^{m_1} 2^{m_2} \cdots r^{m_r}$, the interval $[\Fine_m,\pi]$ in $\Pi_m$ is isomorphic to $\Pi_1^{\times m_1} \times \Pi_2^{\times m_2} \times \cdots \times \Pi_r^{\times m_r}$. Thus, $\widetilde{H}_{\ast}(\Delta(\Fine_m, \pi))$ is concentrated in the top degree $|\nu| - \ell(\nu) - 2$ and, as a vector space, is isomorphic to $\bigotimes_j \widetilde{H}_{j-3}(\Pi_{j})^{\otimes m_j}$. 
Sundaram and Welker~\cite{SW96} described the action of $G_{\nu}$ on this vector space:
%

\begin{theorem}[\cite{SW96}]
  With notation as above, as a $G_{\nu}$-module we have,
  \begin{equation}
    \widetilde{H}_{|\nu| - \ell(\nu) -2} (\Delta(\Fine_m, \pi)) \cong 
    (Q_1 \emph{\S} \One) \otimes (Q_2 \emph{\S} \epsilon) \otimes (Q_3 \emph{\S} \One) \otimes \cdots \label{SWFormula} 
  \end{equation}
  where $\epsilon$ is the sign representation and $Q_j$ is the $\Sym_j$ representation on $\widetilde{H}_{j-3}(\Pi_{j})$.
\end{theorem}

%

Recall the \emph{Whitehouse module} $\Wh_m$ \cite{RW96} is the part of the free Lie algebra on $x_1$, \ldots, $x_{m}$ spanned by commutators of the form $[ \cdots [[x_{w(1)}, x_{w(2)}], x_{w(3)}],\ldots,x_{w(m)}]$ for $w \in \Sym_{m}$. It has dimension $(m-1)!$, with a basis given by those commutators with $w(1)=1$. Brandt \cite{Bra44} showed that the Frobenius character of the Whitehouse module $\Wh_m$ is precisely the Lyndon symmetric function, i.e. $L_m = \mathrm{ch}(\Wh_m)$. 

Stanley~\cite{Sta82}, based on earlier computations of Hanlon \cite{Han81}, showed $\widetilde{H}_i(\Pi_{m})$ relates to the Whitehouse module $\Wh_m$ as follows. 

\begin{theorem}[\cite{Sta82}]
  The $\Sym_{m}$ representation on $\widetilde{H}_i(\Pi_{m})$ vanishes for $i\neq m-3$, and for $i=m-3$ is given by
  \begin{equation}
    Q_m \cong \epsilon \otimes \Wh_{m}
  \label{e:lattice}
  \end{equation}
  where $\epsilon$ is the sign representation of $\Sym_{m}$ and $\Wh_{m}$ is the Whitehouse module.
  \label{thm:lattice}
\end{theorem}

We now give a plethystic formula for the expansion of $M^t_{\mu}$ into the $\SS^{\lambda} (\CC^t)$ basis.

\begin{theorem}
  In the representation ring $\Rep(\Sym_t)$, we have
  \begin{equation}
    [ M^t_{\mu} ] = \sum_{\lambda} (-1)^{|\mu| - |\lambda|} \langle s_{\mu^{T}} , s_{\lambda^{T}} [L_1 + L_2 + L_3 + \cdots ] \rangle [ \SS^{\lambda} (\CC^t)] .
    \label{e:M to S omega}
  \end{equation}
\end{theorem}

\begin{proof}
  Returning to Eq.~\eqref{Distribute Mobius D}, the action of $G_{\nu}$ on $\Sp_{\lambda}$ is via the quotient $G_{\nu} = \prod \Wr(j, m_j) \longrightarrow \prod \Sym_{m_j}$ followed by the inclusion $\prod \Sym_{m_j} \subset \Sym_{\ell}$. The restriction of $\Sp_{\lambda}$ to $\prod \Sym_{m_j}$ is given by Littlewood-Richardson coefficients, namely, in $\Rep(G_{\nu})$, we have
  \[ [\Sp_{\lambda}] = \sum_{|\mu(j)| = m_j} c_{\mu(1) \mu(2) \cdots \mu(r)}^{\lambda} [\Sp_{\mu(1)} \boxtimes \Sp_{\mu(2)} \boxtimes \cdots \boxtimes  \Sp_{\mu(r)} ]  . \]
  We emphasize that each $\mu_{(j)}$ is a partition. Combining this with Eq.~\eqref{SWFormula} gives
  \begin{multline}
    \meq(\nu) \otimes [\Sp_{\lambda}] = \\
    (-1)^{|\nu| - \ell(\nu) - 2} \sum_{|\mu(j)| = m_j} c_{\mu(1) \mu(2) \cdots \mu(r)}^{\lambda} [Q_1 \S \Sp_{\mu(1)} ] \otimes  [Q_2 \S (\epsilon \otimes \Sp_{\mu(2)}) ] \otimes [Q_3 \S \Sp_{\mu(3)} ]\cdots \\
=  (-1)^{|\nu| - \ell(\nu) - 2} \sum_{|\mu(j)| = m_j} c_{\mu(1) \mu(2) \cdots \mu(r)}^{\lambda} [Q_1 \S \Sp_{\mu(1)} ] \otimes  [Q_2 \S \Sp_{\mu(2)^T} ] \otimes [Q_3 \S \Sp_{\mu(3)} ] \cdots 
  \end{multline}
  We now induce from $G_{\nu}$ to $\Sym_m$ using Theorem~\ref{thm:Stanley-Wr}. Setting $q_j = \mathrm{ch}(Q_j)$ gives
  \begin{multline}
    \Ind_{G_{\nu}}^{\Sym_m} \left(\meq(\nu) \otimes [\Sp_{\lambda}]\right) = \\
    (-1)^{|\nu| - \ell(\nu)} \sum_{|\mu_{(j)}| = m_j} c_{\mu_{(1)} \mu_{(2)} \cdots \mu_{(r)}}^{\lambda} s_{\mu_{(1)}}[q_1] \cdot s_{\mu_{(2)}^{T}}[q_2] \cdot s_{\mu_{(3)}}[q_3] \cdots 
  \end{multline}
  Recall the plethysm property (see \cite[Chapter~I.8]{Mac95})
  \[ \omega\left( f[g] \right) = \left\{ \begin{array}{rl}
    f[\omega g] & \mathrm{deg}(g) \text{ even} , \\
    (\omega f)[\omega g] & \mathrm{deg}(g) \text{ odd} .
  \end{array} \right. \]
  Since $Q_j$ is a representation of $\Sym_j$, its character $q_j$ has degree $j$, and so
  \begin{multline}
    \omega \Ind_{G_{\nu}}^{\Sym_m} \left(\meq(\nu) \otimes [\Sp_{\lambda}]\right) = \\
    (-1)^{|\nu| - \ell(\nu)} \sum_{|\mu_{(j)}| = m_j} c_{\mu_{(1)} \mu_{(2)} \cdots \mu_{(r)}}^{\lambda} s_{\mu_{(1)}^{T}}[\omega q_1] \cdot s_{\mu_{(2)}^{T}}[\omega q_2] \cdot s_{\mu_{(3)}^{T}}[\omega q_3] \cdots
  \end{multline}
  Now we must sum over all partitions $\nu$ of size $m$. Recall that $|\lambda|=\ell(\nu)$. Since $\omega$ is an isometry, we have $c_{\mu,\nu}^{\lambda} = c_{\mu^{T},\nu^{T}}^{\lambda^{T}}$. Once again using the plethysm identity $s_{\mu}[f+g] = \sum_{\alpha,\beta} c_{\alpha,\beta}^{\mu} s_{\alpha}[f] s_{\beta}[g]$ (see \cite[I(8.8)]{Mac95}), we have
  \begin{multline}
    \sum_{|\nu|=m} \omega \Ind_{G_{\nu}}^{\Sym_m} \left(\meq(\nu) \otimes [\Sp_{\lambda}]\right) \\
    = \sum_{|\nu|=m} (-1)^{|\nu| - \ell(\nu)} \sum_{|\mu_{(j)}^{T}| = m_j} c_{\mu_{(1)}^{T} \mu_{(2)}^{T} \cdots \mu_{(r)}^{T}}^{\lambda^{T}} s_{\mu_{(1)}^{T}}[\omega q_1] \cdot s_{\mu_{(2)}^{T}}[\omega q_2] \cdot s_{\mu_{(3)}^{T}}[\omega q_3] \cdots \\
    = (-1)^{m - |\lambda|} \sum_{\sum |\mu_{(j)}| = |\lambda^{T}|} c_{\mu_{(1)} \mu_{(2)} \cdots \mu_{(r)}}^{\lambda^{T}} s_{\mu_{(1)}}[\omega q_1] \cdot s_{\mu_{(2)}}[\omega q_2] \cdot s_{\mu_{(3)}}[\omega q_3] \cdots \\
    = (-1)^{m - |\lambda|} s_{\lambda^{T}} [\omega q_1 + \omega q_2 + \omega q_3 + \cdots ] \\
    = (-1)^{m - |\lambda|} s_{\lambda^{T}} [L_1 + L_2 + L_3 + \cdots ] ,
  \end{multline}
  where the last equality follows from taking Frobenius characters in Theorem~\ref{thm:lattice}.
  
  Therefore the coefficient of $[\Sp_{\mu}]$ in $\sum_{\nu} \Ind_{G_{\nu}}^{\Sym_m} (\meq(\nu) \otimes [\Sp_{\lambda}])$ is the coefficient of $s_{\mu}$ in $\omega (-1)^{m - |\lambda|} s_{\lambda^{T}} [\omega q_1 + \omega q_2 + \omega q_3 + \cdots ]$, which is also the coefficient of $\omega s_{\mu} = s_{\mu^{T}}$ in $(-1)^{m - |\lambda|} s_{\lambda^{T}} [\omega q_1 + \omega q_2 + \omega q_3 + \cdots ]$. The theorem now follows from Eq.~\eqref{Schur Weyl for D}, in which $M_{\mu}^t$ can be identified with the $\Sp_{\mu}$-isotypic component of $D(m,t)$ where $m = |\mu|$, and Eq.~\eqref{Distribute Mobius D}, in which the coefficient of $[\SS_{\lambda}(\CC^t)]$ in $[M_{\mu}^t]$ is the coefficient of $[\Sp_{\mu}]$ in $\sum_{\nu} \Ind_{G_{\nu}}^{\Sym_m} (\meq(\nu) \otimes [\Sp_{\lambda}])$. 
\end{proof}

Thus we have proved the final part of Theorem~\ref{thm:MT FS}. Combining the second and fourth parts of Theorem~\ref{thm:MT FS} proves Theorem~\ref{thm:MT-plethysm}.

\section*{Acknowledgments}

The authors thank Jon Schneider for implementing the characters and change of basis matrices between them in Java, and John Wiltshire-Gordon for related insights into the category of representations of finite sets. We also thank Rosa Orellana and Mike Zabrocki for interesting conversations about the symmetric function approach to understanding stable Kronecker coefficients through the stable Specht functions, which they discovered independently.


\begin{table}[ht]
  \begin{displaymath}
    \renewcommand{\arraystretch}{1.25}
    \begin{array}{>{\scriptstyle}l@{\hskip 5pt}>{\scriptstyle}c@{\hskip 5pt}>{\scriptstyle}l}
      s_{(1)}         & = & 1+\sdag_{(1)} \\ 
      s_{(2)}         & = & 2+2\sdag_{(1)}+\sdag_{(2)} \\ 
      s_{(11)}       & = & \sdag_{(1)}+\sdag_{(11)} \\ 
      s_{(3)}         & = & 3+4\sdag_{(1)}+2\sdag_{(2)}+\sdag_{(11)}+\sdag_{(3)} \\ 
      s_{(21)}       & = & 1+3\sdag_{(1)}+2\sdag_{(2)}+2\sdag_{(11)}+\sdag_{(21)} \\ 
      s_{(111)}     & = & \sdag_{(11)}+\sdag_{(111)} \\ 
      s_{(4)}         & = & 5+7\sdag_{(1)}+5\sdag_{(2)}+2\sdag_{(11)}+2\sdag_{(3)}+\sdag_{(21)}+\sdag_{(4)} \\ 
      s_{(31)}       & = & 2+7\sdag_{(1)}+5\sdag_{(2)}+6\sdag_{(11)}+2\sdag_{(3)}+3\sdag_{(21)}+\sdag_{(111)}+\sdag_{(31)} \\ 
      s_{(22)}       & = & 2+3\sdag_{(1)}+4\sdag_{(2)}+\sdag_{(11)}+\sdag_{(3)}+2\sdag_{(21)}+\sdag_{(22)} \\ 
      s_{(211)}     & = & \sdag_{(1)}+\sdag_{(2)}+3\sdag_{(11)}+2\sdag_{(21)}+2\sdag_{(111)}+\sdag_{(211)} \\ 
      s_{(1111)}   & = & \sdag_{(111)}+\sdag_{(1111)} \\ 
      s_{(5)}         & = & 7+12\sdag_{(1)}+9\sdag_{(2)}+5\sdag_{(11)}+5\sdag_{(3)}+3\sdag_{(21)}+2\sdag_{(4)}+\sdag_{(31)}+\sdag_{(5)} \\ 
      s_{(41)}       & = & 5+14\sdag_{(1)}+13\sdag_{(2)}+12\sdag_{(11)}+6\sdag_{(3)}+9\sdag_{(21)}+3\sdag_{(111)}+2\sdag_{(4)}+3\sdag_{(31)}+\sdag_{(22)}+\sdag_{(211)}+\sdag_{(41)} \\ 
      s_{(32)}       & = & 4+10\sdag_{(1)}+11\sdag_{(2)}+8\sdag_{(11)}+6\sdag_{(3)}+8\sdag_{(21)}+2\sdag_{(111)}+\sdag_{(4)}+3\sdag_{(31)}+2\sdag_{(22)}+\sdag_{(211)}+\sdag_{(32)} \\ 
      s_{(311)}     & = & 3\sdag_{(1)}+4\sdag_{(2)}+8\sdag_{(11)}+\sdag_{(3)}+7\sdag_{(21)}+6\sdag_{(111)}+2\sdag_{(31)}+\sdag_{(22)}+3\sdag_{(211)}+\sdag_{(1111)}+\sdag_{(311)} \\ 
      s_{(221)}     & = & 1+3\sdag_{(1)}+4\sdag_{(2)}+3\sdag_{(11)}+2\sdag_{(3)}+5\sdag_{(21)}+\sdag_{(111)}+\sdag_{(31)}+2\sdag_{(22)}+2\sdag_{(211)}+\sdag_{(221)} \\ 
      s_{(2111)}   & = & \sdag_{(11)}+\sdag_{(21)}+3\sdag_{(111)}+2\sdag_{(211)}+2\sdag_{(1111)}+\sdag_{(2111)} \\ 
      s_{(11111)} & = & \sdag_{(1111)}+\sdag_{(11111)}
    \end{array}
  \end{displaymath}
  \caption{\label{tab:s2sdag}Schur functions expanded into the stable Specht basis.}
\end{table}

\begin{table}[ht]
  \begin{displaymath}
    \renewcommand{\arraystretch}{1.25}
    \begin{array}{>{\scriptstyle}l@{\hskip 5pt}>{\scriptstyle}c@{\hskip 5pt}>{\scriptstyle}l}
\sdag_{(1)} & = & s_{(1)}-1 \\ 
\sdag_{(2)} & = & s_{(2)}-2s_{(1)} \\ 
\sdag_{(1 1)} & = & s_{(11)}-s_{(1)}+1 \\ 
\sdag_{(3)} & = & s_{(3)}-2s_{(2)}-s_{(11)}+s_{(1)} \\ 
\sdag_{(2 1)} & = & s_{(21)}+3s_{(1)}-2s_{(2)}-2s_{(11)} \\ 
\sdag_{(1 1 1)} & = & s_{(111)}-s_{(11)}+s_{(1)}-1 \\ 
\sdag_{(4)} & = & s_{(4)}+s_{(2)}-2s_{(3)}-s_{(21)}+2s_{(11)} \\ 
\sdag_{(3 1)} & = & s_{(31)}-3s_{(1)}+5s_{(2)}+3s_{(11)}-2s_{(3)}-3s_{(21)}-s_{(111)} \\ 
\sdag_{(2 2)} & = & s_{(22)}-s_{(1)}+2s_{(2)}-s_{(3)}-2s_{(21)}+4s_{(11)} \\ 
\sdag_{(2 1 1)} & = & s_{(211)}-4s_{(1)}+3s_{(2)}+3s_{(11)}-2s_{(21)}-2s_{(111)} \\ 
\sdag_{(1 1 1 1)} & = & s_{(1111)}-s_{(111)}+s_{(11)}-s_{(1)}+1 \\ 
\sdag_{(5)} & = & s_{(5)}+s_{(3)}+2s_{(21)}-s_{(11)}-2s_{(4)}-s_{(31)}+s_{(111)} \\ 
\sdag_{(4 1)} & = & s_{(41)}+2s_{(1)}-5s_{(11)}+5s_{(3)}+6s_{(21)}-2s_{(4)}-3s_{(31)}-s_{(22)}-5s_{(2)}-s_{(211)}+2s_{(111)} \\ 
\sdag_{(3 2)} & = & s_{(32)}+3s_{(1)}-6s_{(2)}-6s_{(11)}+4s_{(3)}+8s_{(21)}+3s_{(111)}-s_{(4)}-3s_{(31)}-2s_{(22)}-s_{(211)} \\ 
\sdag_{(3 1 1)} & = & s_{(311)}+5s_{(1)}-7s_{(11)}+4s_{(3)}+7s_{(21)}+3s_{(111)}-2s_{(31)}-s_{(22)}-3s_{(211)}-9s_{(2)}-s_{(1111)} \\ 
\sdag_{(2 2 1)} & = & s_{(221)}+3s_{(1)}-7s_{(11)}+6s_{(21)}-5s_{(2)}-s_{(31)}-2s_{(22)}+2s_{(3)}-2s_{(211)}+4s_{(111)} \\ 
\sdag_{(2 1 1 1)} & = & s_{(2111)}-4s_{(11)}+5s_{(1)}+3s_{(21)}-4s_{(2)}+3s_{(111)}-2s_{(211)}-2s_{(1111)} \\ 
\sdag_{(1 1 1 1 1)} & = & s_{(11111)}-s_{(1111)}+s_{(111)}-s_{(11)}+s_{(1)}-1 
    \end{array}
  \end{displaymath}
  \caption{\label{tab:sdag2s}Stable Specht functions expanded into the Schur basis.}
\end{table}

%
%

\bibliographystyle{plain} 
\bibliography{alt}

\end{document}